\documentclass[12pt, reqno]{amsart}
 \usepackage{amsmath, amsthm, amscd, amsfonts, amssymb, graphicx, color, float}
\usepackage[bookmarksnumbered, colorlinks, plainpages]{hyperref}

\newtheorem{theorem}{Theorem}[section]

\newtheorem{corollary}[theorem]{Corollary}
\theoremstyle{definition}

\theoremstyle{remark}

\numberwithin{equation}{section}

\begin{document}

\title{Noncommutative martingale concentration inequalities}
\author[Gh. Sadeghi, M.S. Moslehian]{Ghadir Sadeghi$^1$ and Mohammad Sal Moslehian$^2$}

\address{$^1$ Department of Mathematics and Computer
Sciences, Hakim Sabzevari University, P.O. Box 397, Sabzevar, Iran}
\email{ghadir54@gmail.com, g.sadeghi@hsu.ac.ir}

\address{$^2$ Department of Pure Mathematics, Center of Excellence in
Analysis on Algebraic Structures (CEAAS), Ferdowsi University of
Mashhad, P.O. Box 1159, Mashhad 91775, Iran}
\email{moslehian@um.ac.ir, moslehian@member.ams.org}


\begin{abstract}
We establish an Azuma type inequality under a Lipshitz condition for martingales in the framework of noncommutative probability spaces and apply it to deduce a noncommutative Heoffding inequality as well as a noncommutative McDiarmid type inequality. We also provide a noncommutative Azuma inequality for noncommutative supermartingales in which instead of a fixed upper bound for the variance we assume that the variance is bounded above by a linear function of variables. We then employ it to deduce a noncommutative Bernstein inequality and an inequality involving $L_p$-norm of the sum of a martingale difference.

\end{abstract}

\maketitle

\section{Introduction and preliminaries}

In probability theory, inequalities giving upper bounds on ${\rm Prob}(|X-\mathbb{E}(X)|)$, where $X$ is a random variable and $\mathbb{E}(X)$ denotes its expectation are of special interest, see \cite{ANA, TAL, 3}. Among such inequalities, the Azuma inequality, due to K. Azuma \cite{AZU}, provides a concentration result for the values of martingales having bounded differences. It states that if $(X_j)$ is a martingale and $|X_j-X_j-1|<c_j$ almost surely, then $${\rm Prob}(X_n-X_0\geq \lambda)\leq \exp\left(-\lambda^2/(2\sum_{j=1}^nc_j^2)\right)$$ for all positive integers $n$ and all $\lambda>0$. This inequality can be employed to the study of random graphs, see \cite{FRI}.
In this paper we establish an Azuma type inequality under a Lipshitz condition for martingales in the framework of noncommutative probability spaces and apply it to deduce a noncommutative Heoffding inequality as well as a noncommutative McDiarmid type bounded difference inequality; see \cite{MCD}. We also provide a noncommutative Azuma inequality for noncommutative supermartingales in which instead of a fixed upper bound for the variance we assume that the variance is bounded above by a linear function of variables. We then employ it to deduce a noncommutative Bernstein inequality, which gives an upper bound on the probability that the sum of independent random variables is more than a fixed amount, and an inequality involving $L_p$-norm of the sum of a martingale difference, see also \cite{PQ}. To achieve our goal we first fix our notation and terminology.

A von Neumann algebra $\mathfrak{M}$ on a Hilbert space with unit element $1$ equipped with a normal faithful tracial state $\tau:\mathfrak{M}\to \mathbb{C}$ is called a noncommutative probability space. We denote by $\leq$ the usual order on self-adjoint part $\mathfrak{M}^{sa}$ of $\mathfrak{M}$.
For each self-adjoint operator $x\in \mathfrak{M}$, there exists a unique spectral measure $E$ as a $\sigma$-additive mapping with respect to the strong operator topology from the Borel $\sigma$-algebra $\mathcal{B}(\mathbb{R})$ of $\mathbb{R}$ into the set of all orthogonal projections such that for every Borel function $f: \sigma(x)\to \mathbb{C}$ the operator $f(x)$ is defined by $f(x)=\int f(\lambda)dE(\lambda)$, in particular, $\chi_B(x)=\int_BdE(\lambda)=E(B)$. Of course, the modules $|x|$ of $x\in \mathfrak{M}$ can be defined by $|x|=(x^*x)^{1/2}$ by utilizing the usual functional calculus.
The inequality
\begin{eqnarray}\label{GH1}
{\rm Prob}(x\geq t):=\tau(\chi_{[t,\infty)}(x))\leq e^{-t}\tau(e^x)\,.
\end{eqnarray}
is known as exponential Chebyshev inequality in the literature.
The celebrated Golden--Thompson inequality \cite{RUS} (see also \cite{FT}) states that for any self-adjoint elements $y_1, y_2$ in a noncommutative probability space $ \mathfrak{M}$,
\begin{eqnarray}\label{T1}
\tau(e^{y_1+y_2})\leq \tau(e^{y_1/2}e^{y_2}e^{y_1/2})\,
\end{eqnarray}\label{T2}
and
\begin{eqnarray}\label{T2}
\tau(e^{y_1+y_2})\leq \tau(e^{y_1} e^{y_2}).
\end{eqnarray}

For $p\geq1$, the noncommutative $L_p$-space $L_p(\mathfrak{M})$ is defined as the completion of $\mathfrak{M}$ with respect to the $L_p$-norm $\|x\|_p:=\left(\tau(|x|^p)\right)^{1/p}$. Further, for a positive element $x\in\mathfrak{M}$, it holds that
\begin{eqnarray}\label{I}
\|x\|_p^p=\int_0^{\infty} pt^{p-1}\tau(\chi_{[t,\infty)}(x))dt.
\end{eqnarray}
The commutative cases of discussed spaces are usual $L^p$-spaces and the Schatten $p$-classes $\mathcal{C}_p$. For further information we refer the reader to \cite{P, MTS} and references therein.

Let $\mathfrak{N}$ be a von Neumann subalgebra of $\mathfrak{M}$. Then there exists a normal contraction positive mapping projecting $\mathcal{E}_{\mathfrak{N}}:\mathfrak{M}\to\mathfrak{N}$ satisfying the following properties:\\
(i) $\mathcal{E}_{\mathfrak{N}}(axb) =a\mathcal{E}_{\mathfrak{N}}(x)b$ for any $x\in\mathfrak{M}$ and $a, b\in\mathfrak{N}$;\\
(ii) $\tau\circ\mathcal{E}_{\mathfrak{N}}=\tau$.\\
Moreover, $\mathcal{E}_{\mathfrak{N}}$ is the unique mapping satisfying (i) and (ii). The mapping $\mathcal{E}_{\mathfrak{N}}$ is called the conditional expectation of $\mathfrak{M}$ with respect to $\mathfrak{N}$.

Let $\mathfrak{N}\subseteq \mathfrak{A}_j\,\,(1\leq j\leq n)$ be von Neumann subalgebras of $\mathfrak{M}$.
We say that the $\mathfrak{A}_j$ are order independent over $\mathfrak{N}$ if for every $2\leq j\leq n$, the equality $$\mathcal{E}_{j-1}(x)=\mathcal{E}_{\mathfrak{N}}(x)$$ holds for all $x\in\mathfrak{A}_j$, where $\mathcal{E}_{j-1}$ is the conditional expectation of $\mathfrak{M}$ with respect to the von Neumann subalgebra generated by $\mathfrak{A}_1,\ldots,\mathfrak{A}_{j-1}$; cf. \cite{JX}.

A filtration of $\mathfrak{M}$ is an increasing sequence $(\mathfrak{M}_j, \mathcal{E}_j)_{0\leq j\leq n}$ of von Neumann subalgebras of $\mathfrak{M}$ together with the conditional expectations $ \mathcal{E}_j$ of $\mathfrak{M}$ with respect to $\mathfrak{M}_j$ such that $\bigcup_j\mathfrak{M}_j$ is $w^*$--dense in $\mathfrak{M}$. It follows from $\mathfrak{M}_j\subseteq \mathfrak{M}_{j+1}$ that
\begin{eqnarray}\label{maral}
\mathcal{E}_i\circ\mathcal{E}_j=\mathcal{E}_j\circ\mathcal{E}_i=\mathcal{E}_{\min\{i,j\}}\,.
\end{eqnarray}
for all $i,j\geq 0$. A finite sequence $(x_j)_{0\leq j\leq n}$ in $L^1(\mathfrak{M})$ is called a martingale (supermartingale, resp.) with respect to filtration $(\mathfrak{M}_j)_{0\leq j\leq n}$ if $x_j \in \mathfrak{M}_j$ and $\mathcal{E}_j(x_{j+1})=x_j$ ($\mathcal{E}_j(x_{j+1})\leq x_j$, resp.) for every $j\geq 0$. It follows from \eqref{maral} that $\mathcal{E}_j(x_i)=x_j$ for all $i\geq j$, in particular $x_j=\mathcal{E}_j(x_n)$ for all $0\leq j\leq n$, in other words, each martingale can be adopted by an element. Put $dx_j=x_j-x_{j-1}\,\,(j\geq 0)$ with the convention that $x_{-1}=0$. Then $dx=(dx_j)_{n\geq 0}$ is called the martingale difference of $(x_j)$. The reader is referred to \cite{Xu1, Xu2} for more information.


\section{noncommutative Azuma inequality subject to a Lipschitz condition}

In this section we provide a noncommutative Azuma inequality under a Lipschitz condition.
\begin{theorem} (Noncommutative Azuma inequality)\label{main1}
Let $x=(x_j)_{0\leq j\leq n}$ be a self-adjoint martingale with respect to a filtration $(\mathfrak{M}_j, \mathcal{E}_j)_{0\leq j\leq n}$ and $dx_j=x_j-x_{j-1}$ be its associated martingale difference. Assume that $-c_j\leq dx_j\leq c_j$ for some constants $c_j>0\,\, (1\leq j\leq n)$. Then
\begin{eqnarray}\label{MOS2}
{\rm Prob}\left(\left|\sum_{j=1}^ndx_j\right|\geq \lambda\right)\leq 2 \exp\left\{\frac{-\lambda^2}{2\sum_{j=1}^nc_j^2}\right\}
\end{eqnarray}
for all $\lambda > 0$.
\end{theorem}
\begin{proof}
For a fixed number $t>0$, we consider the convex function $f(s)=e^{ts}$. It follows from the convexity of $f$ that
\begin{eqnarray*}
e^{ts}\leq \frac{1}{2c}(e^{tc}-e^{-tc})s+\frac{1}{2}(e^{tc}+e^{-tc})
\end{eqnarray*}
for any $-c\leq s\leq c$.\\
Since $-c_j\leq dx_j\leq c_j$, by the functional calculus, we have
\begin{eqnarray*}
e^{tdx_j}\leq \frac{1}{2c_j}(e^{tc_j}-e^{-tc_j})dx_j+\frac{1}{2}(e^{tc_j}+e^{-tc_j}).
\end{eqnarray*}
Hence
\begin{eqnarray*}
\mathcal{E}_{j-1}\left(e^{tdx_j}\right)&\leq&\mathcal{E}_{j-1}\left(\frac{1}{2c_j}(e^{tc_j}-e^{-tc_j})dx_j+\frac{1}{2}(e^{tc_j}+e^{-tc_j})\right)\\
&=&\frac{1}{2}(e^{tc_j}+e^{-tc_j})\qquad\qquad\qquad (\mbox{by~} \mathcal{E}_{j-1}(dx_j)=0, j\geq 2)\\
&=&\sum_{n=0}^\infty\frac{(tc_j)^{2n}}{(2n)!}\leq \sum_{n=0}^\infty\frac{(tc_j)^{2n}}{2^nn!}= e^{\frac{t^2c_j^2}{2}}\,.
\end{eqnarray*}
Now by inequality \eqref{GH1}, for $\lambda \geq0$, we have
\begin{eqnarray*}
{\rm Prob}\left(\sum_{j=1}^ndx_j\geq \lambda\right)&\leq& e^{-t\lambda}\tau\left(e^{t\sum_{j=1}^ndx_j}\right)\\
&\leq&e^{-t\lambda}\tau\left(e^{t\sum_{j=1}^{n-1}dx_j}e^{tdx_n}\right)\\
&=&e^{-t\lambda}\tau\left(\mathcal{E}_{n-1}\left(e^{t\sum_{j=1}^{n-1}dx_j}e^{tdx_n}\right)\right)\\
&=&e^{-t\lambda}\tau\left(e^{t\sum_{j=1}^{n-1}dx_j}\mathcal{E}_{n-1}\left(e^{tdx_n}\right)\right)\\
&\leq&e^{-t\lambda}e^{t^2c_n^2/2}\tau\left(e^{t\sum_{j=1}^{n-1}dx_j}\right)
\end{eqnarray*}
Iterating $n-2$ times, we obtain
\begin{eqnarray*}
{\rm Prob}\left(\sum_{j=1}^ndx_j\geq \lambda\right)\leq\exp\left(-t\lambda+\frac{t^2}{2}\sum_{j=1}^nc_j^2\right).
\end{eqnarray*}
It is easy to see that the the minimizing value of $\exp\left(-t\lambda+\frac{t^2}{2}\sum_{j=1}^nc_j^2\right)$ occurs at $t=\frac{\lambda}{\sum_{j=1}^nc_j^2}$. So
\begin{eqnarray}\label{A1}
{\rm Prob}\left(\sum_{j=1}^ndx_j\geq \lambda\right)\leq\exp\left(\frac{-\lambda^2}{2\sum_{j=1}^nc_j}\right).
\end{eqnarray}
Therefore symmetry and inequality (\ref{A1}) imply that
\begin{eqnarray*}
{\rm Prob}\left(\left|\sum_{j=1}^ndx_j\right|\geq \lambda\right)=2{\rm Prob}\left(\sum_{j=1}^ndx_j\geq \lambda\right)\leq2\exp\left(\frac{-\lambda^2}{2\sum_{j=1}^nc_j}\right).
\end{eqnarray*}
\end{proof}

The first consequence reads as follows.

\begin{corollary}(Noncommutative Hoeffding inequality)\label{H}
Let $\mathfrak{N}\subseteq\mathfrak{A}_j(\subseteq\mathfrak{M})$ be order independent over $\mathfrak{N}$. Let $x_j\in\mathfrak{A}_j$ be self-adjoint such that $\mathcal{E}_{\mathfrak{N}}(x_j)=0$ and $-c_j\leq x_j\leq c_j$ for some constants $c_j>0\,\,(1\leq j\leq n)$. Then
\begin{eqnarray}
{\rm Prob}\left(\left|S_n\right|\geq t\right)\leq 2 \exp\left\{\frac{-t^2}{2\sum_{j=1}^n c_j^2}\right\}.
\end{eqnarray}
for any $t>0$, where $S_n=\sum_{j=1}^n x_j$.
\end{corollary}
\begin{proof}
Let $\mathfrak{M}_0=\mathfrak{N}$ and $\mathcal{E}_0=\mathcal{E}_\mathfrak{N}$. For every $1\leq j\leq n$, let $\mathfrak{M}_j$ be the von Neumann subalgebra generated by $\mathfrak{A}_1,\ldots,\mathfrak{A}_{j-1}$ and $\mathcal{E}_j$ be the corresponding conditional expectation. Put $S_0:=0$ and $S_j:=\sum_{k=1}^jx_k$ for $1\leq j\leq n$. Then
$$\mathcal{E}_{j-1}(S_j)=\sum_{k=1}^{j-1}x_k+\mathcal{E}_{j-1}(x_k)=
\sum_{k=1}^{j-1}x_k+\mathcal{E}_{\mathfrak{N}}(x_k)=S_{j-1}$$
So $(S_j)_{0\leq j\leq n}$ is a martingale with respect to filtration $(\mathfrak{M}_j, \mathcal{E}_j)_{0\leq j\leq n}$. Since
$$dS_j=\sum_{k=1}^jx_k-\sum_{k=1}^{j-1}x_k=x_j$$
the required inequality follows from Theorem \ref{main1}.
\end{proof}

The next results present some noncommutative McDiarmid type inequalities.

\begin{corollary} (Noncommutative McDiarmid inequality)\label{mmcc}
Let $(\mathfrak{M}_j, \mathcal{E}_j)_{0\leq j\leq n}$ be a filtration of $\mathfrak{M}$, $x_j\in \mathfrak{M}_j^{sa}\,\,(1\leq j\leq n)$ and there exist mappings $g_j:\mathfrak{M}_1^{sa}\times \cdots\times \mathfrak{M}_j^{sa}\to\mathfrak{M}^{sa}$ such that the sequence $g_0(x_1,\ldots,x_n)), g_1(x_1, \cdots, x_n), \cdots, g_n(x_1, \cdots, x_n)$ constitute a martingale satisfying
\begin{eqnarray*}
-c_j\leq g_j(x_1, \cdots, x_n)-g_{j-1}(x_1, \cdots, x_n)\leq c_j
\end{eqnarray*}
for any $1\leq j\leq n$. Then
\begin{eqnarray}\label{MOS2}
{\rm Prob}\left(\left| g_n(x_1,\ldots,x_n)-g_0(x_1,\ldots,x_n))\right|\geq t\right)\leq 2 \exp\left\{\frac{-t^2}{2\sum_{j=1}^nc_j^2}\right\}.
\end{eqnarray}
\end{corollary}
\begin{proof}
The result can be deduced immediately from Theorem \ref{main1} due to the martingale consisting of $y_j=g_j(x_1, \cdots, x_n),\,0\leq j\leq n$ satisfies the conditions of the theorem.
\end{proof}

Considering $c_j=1$ and $g_j(X_1, \cdots, X_n)=\sum_{i=1}^j X_i$ in the previous Corollary, we reach the following Chernoff type inequality for random variables:
\begin{corollary}
Let $X_1, \cdots, X_n$ be independent random variables with $\mathbb{E}(X_j)=0$ and $|X_j|\leq 1$ for all $j$. Then
$${\rm Prob}\left(\left|\sum_{j=1}^nX_j\right|\geq t\right)\leq 2e^{-t^2/2n}\,.$$
for all $t \geq 0$.
\end{corollary}

The following is another version of the noncommutative McDiarmid inequality.

\begin{corollary}\label{corfir}
Let $\mathfrak{N}\subseteq\mathfrak{M}$ and $(\mathfrak{M}_j, \mathcal{E}_j)_{0\leq j\leq n}$ be a filtration of $\mathfrak{M}$, $\mathfrak{M}_0=\mathfrak{N}$ and there be a mapping $g:\mathfrak{M}_1^{sa}\times \cdots\times \mathfrak{M}_n^{sa}\to\mathfrak{M}^{sa}$ and elements $x_j\in\mathfrak{M}_j^{sa}\,\,(1\leq j\leq n)$ such that
\begin{eqnarray*}
-c_j\leq \mathcal{E}_j(g(x_1,\ldots,x_n))-\mathcal{E}_{j-1}(g(x_1,\ldots,x_n))\leq c_j
\end{eqnarray*}
for any $1\leq j\leq n$. Then
\begin{eqnarray}\label{MOS2}
{\rm Prob}\left( \left|g(x_1,\ldots,x_n)-\mathcal{E}_{\mathfrak{N}}(g(x_1,\ldots,x_n))\right|\geq t\right)\leq 2 \exp\left\{\frac{-t^2}{2\sum_{j=1}^nc_j^2}\right\},
\end{eqnarray}
\end{corollary}
\begin{proof}
Let us put $g_n(x_1,\ldots,x_n)=g(x_1,\ldots,x_n)$. Then $g_j(x_1,\ldots,x_n)=\mathcal{E}_j(g(x_1,\ldots,x_n))$ for $0\leq j\leq n$ and we get the martingale $(g_j(x_1,\ldots,x_n))_{0\leq j\leq n}$ with respect to the filtration $(\mathfrak{M}_j, \mathcal{E}_{j})_{0\leq j\leq n}$, which satisfies the conditions Corollary \ref{mmcc}.
\end{proof}


\begin{corollary}
Let $\mathfrak{N}\subseteq\mathfrak{M}$ and $(\mathfrak{M}_j, \mathcal{E}_j)_{0\leq j\leq n}$ be a filtration of $\mathfrak{M}$, $\mathfrak{M}_0=\mathfrak{N}$ and self-adjoint elements $x_j\in\mathfrak{M}_j\,\,(0\leq j\leq n)$ constitute a martingale with respect to $(\mathfrak{M}_j, \mathcal{E}_j)_{0\leq j\leq n}$ and
\begin{eqnarray*}
\frac{-c_j}{n-j+1}\leq x_j-x_{j-1}\leq \frac{c_j}{n-j+1}
\end{eqnarray*}
for any $1\leq j\leq n$. Then
\begin{eqnarray}\label{MOS2}
{\rm Prob}\left(\left| \sum_{k=i}^nx_k-\sum_{k=i}^n\mathcal{E}_\mathfrak{N}(x_{k-1})\right|\geq t\right)\leq 2 e^{\frac{-t^2}{2\sum_{j=1}^nc_j^2}}
\end{eqnarray}
for all $0\leq i\leq j-2$.
\end{corollary}
\begin{proof}
Recall that if $(x_j)_{0\leq j\leq n}$ is a martingale with respect to $(\mathfrak{M}_j, \mathcal{E}_j)_{0\leq j\leq n}$. Hence
$x_{j}=\mathcal{E}_{j}(x)$ for some $x\in \mathfrak{M}$ and all $0\leq j\leq n$. For any $0\leq i\leq j-2$, define the function $g_i$ on $\mathfrak{M}_1\times \cdots\times \mathfrak{M}_n$ by $g_i(y_1,\ldots,y_n):=\sum_{k=i}^ny_k$. Then
$$\mathcal{E}_j\left(g_i(x_1,\ldots,x_n)\right)=\mathcal{E}_j\left(\sum_{k=i}^nx_k\right)=\sum_{k=i}^n \mathcal{E}_{\min\{j,k\}}(x)\,.$$
Hence
$$\mathcal{E}_j\left(g_i(x_1,\ldots,x_n)\right)-\mathcal{E}_{j-1}\left(g_i(x_1,\ldots,x_n)\right)=(n-j+1)(x_j-x_{j-1})\,.$$
Now the requested inequality can be concluded from Corollary \ref{corfir}.
\end{proof}


\section{noncommutative Azuma inequality for supermartingales}

Sometimes Lipschitz conditions seem to be too strong. So we may need some more effective tools.
In the sequel, we prove an extension of the Azuma inequality under some mild conditions. Our first result
is indeed a noncommutative Azuma inequality involving supermartingales. Our approach is based on standard arguments in probability theory \cite{CL}.

\begin{theorem}\label{main3}
Let $x=(x_j)_{0\leq j\leq n}$ be a self-adjoint supermartingale with respect to a filtration $(\mathfrak{M}_j, \mathcal{E}_j)_{0\leq j\leq n}$ such that
for some positive constants $a_j, b_j, \sigma_j$ and $M$ satisfies
\begin{itemize}
\item [(i)] $\mathcal{E}_{j-1}((x_j-\mathcal{E}_{j-1}(x_j))^2)\leq\sigma_j^2+b_jx_{j-1}$,
\item [(ii)] $x_j-\mathcal{E}_{j-1}(x_j)\leq a_j+M$
\end{itemize}
for all $1\leq j\leq n$. Then
\begin{eqnarray}\label{MOS2}
{\rm Prob}\left( x_n -x_0 \geq \lambda\right)\leq \exp\left\{\frac{-\lambda^2}{2\left(\sum_{j=1}^n(\sigma_j^2+Db_j+a_j^2)+(M\lambda/3)\right)}\right\}\,.
\end{eqnarray}
for all $\lambda > 0$, where $D:=\max_{1\leq j\leq n-1}M_j$ and $M_j$ is the maximum of spectrum of $x_j-x_0$.
\end{theorem}
\begin{proof}
\noindent \textbf{Step (I). To prove the theorem in a special case}\\
We assume that $x=(x_j)_{0\leq j\leq n}$ is a supermartingale with $x_0=0$.

\noindent \textbf{Step (II). To find an upper bound for $\tau\left(e^{tx_j}\right)$:}\\
Let $t>0$. We have
\begin{eqnarray}\label{sal}
\tau\left(e^{tx_j}\right)&=&\tau\left(e^{t\mathcal{E}_{j-1}(x_j)+
ta_j+t(x_j-\mathcal{E}_{j-1}(x_j)-a_j)}\right)\nonumber\\
&\leq&\tau\left(e^{\frac{t\mathcal{E}_{j-1}(x_j)+ta_j}{2}}
e^{t(x_j-\mathcal{E}_{j-1}(x_j)-a_j)}e^{\frac{t\mathcal{E}_{j-1}(x_j)+ta_j}{2}}\right)\qquad\qquad(\mbox{by~} \eqref{T1})\nonumber\\
&=&\tau\left(\mathcal{E}_{j-1} \left(e^{\frac{t\mathcal{E}_{j-1}(x_j)+ta_j}{2}}
e^{t(x_j-\mathcal{E}_{j-1}(x_j)-a_j)}e^{\frac{t\mathcal{E}_{j-1}(x_j)+ta_j}{2}}\right)\right)\nonumber\\
&&\qquad\qquad\qquad\qquad\qquad(\mbox{by property (ii) of conditional expectation})\nonumber\\
&=&\tau\left(e^{\frac{t\mathcal{E}_{j-1}(x_j)+ta_j}{2}}
\mathcal{E}_{j-1}\left(e^{t(x_j-\mathcal{E}_{j-1}(x_j)-a_j)}\right)e^{\frac{t\mathcal{E}_{j-1}(x_j)+ta_j}{2}}\right)\nonumber\\
&&\qquad\qquad\qquad\qquad\qquad(\mbox{by property (i) of conditional expectation})\nonumber\\
&=&\tau\left(e^{\frac{t\mathcal{E}_{j-1}(x_j)+ta_j}{2}}
\mathcal{E}_{j-1}\left(\sum_{k=0}^{\infty}\frac{t^k}{k!}(x_j-\mathcal{E}_{j-1}(x_j)-a_j)^k\right)e^{\frac{t\mathcal{E}_{j-1}(x_j)+ta_j}{2}}\right)\nonumber\\
&=&\tau\left(e^{\frac{t\mathcal{E}_{j-1}(x_j)+ta_j}{2}}
\sum_{k=0}^{\infty}\frac{t^k}{k!}\mathcal{E}_{j-1}\left((x_j-\mathcal{E}_{j-1}(x_j)-a_j)^k\right)e^{\frac{t\mathcal{E}_{j-1}(x_j)+ta_j}{2}}\right)\nonumber\\
&\leq&\tau\left(e^{\frac{t\mathcal{E}_{j-1}(x_j)+ta_j}{2}}
e^{\sum_{k=1}^{\infty}\frac{t^k}{k!}\mathcal{E}_{j-1}\left((x_j-\mathcal{E}_{j-1}(x_j)-a_j)^k\right)}e^{\frac{t\mathcal{E}_{j-1}(x_j)+ta_j}{2}}\right)\nonumber\\
&&\qquad \quad(\mbox{by the validity of~} 1+x \leq e^x \mbox{~for any self-adjoint element~} x\in\mathfrak{M})\nonumber\\
&=&\tau\left(e^{\frac{t\mathcal{E}_{j-1}(x_j)+ta_j}{2}}
e^{-ta_j+\sum_{k=2}^{\infty}\frac{t^k}{k!}\mathcal{E}_{j-1}\left((x_j-\mathcal{E}_{j-1}(x_j)-a_j)^k\right)}e^{\frac{t\mathcal{E}_{j-1}(x_j)+ta_j}{2}}\right)\nonumber\\
&&\qquad \qquad\quad(\mbox{by~}\mathcal{E}_{j-1}\left(x_j-\mathcal{E}_{j-1}(x_j)\right)=\mathcal{E}_{j-1}(x_j)-\mathcal{E}_{j-1}(x_j)=0).
\end{eqnarray}
\noindent \textbf{Step (III). To give an upper bound for $\sum_{j=2}^{\infty}\frac{t^k}{k!}\mathcal{E}_{j-1}\left((x_j-\mathcal{E}_{j-1}(x_j)-a_j)^k\right)$.}\\
Put $h(s)=2\sum_{k=2}^{\infty}\frac{s^{k-2}}{k!}$. The function $h$ satisfies (i) $h(s)\leq 1$ for $s\leq 0$ and (ii) $h$ is monotone increasing on $[0,\infty)$. Hence if $s<M$, then
\begin{eqnarray}\label{MOS4}
h(s)\leq \left\{
\begin{array}{ll}
h(M)& \mbox{when~} s\geq 0 \\
1=h(0)\leq h(M)& \mbox{when~} s< 0
\end{array}\right.
\end{eqnarray}
We have
\begin{eqnarray}\label{n1}
&&\hspace{-1cm}\sum_{j=2}^{\infty}\frac{t^k}{k!}\mathcal{E}_{j-1}\left((x_j-\mathcal{E}_{j-1}(x_j)-a_j)^k\right)\nonumber\\
&=&\mathcal{E}_{j-1}\left(\sum_{j=2}^{\infty}\frac{t^k}{k!}(x_j-\mathcal{E}_{j-1}(x_j)-a_j)^k\right)\nonumber\\
&=& \mathcal{E}_{j-1}\left(\frac{t^2}{2} (x_j-\mathcal{E}_{j-1}(x_j)-a_j)^2h\left(t\left(x_j-\mathcal{E}_{j-1}(x_j)-a_j\right)\right)\right)\nonumber\\
&\leq& \mathcal{E}_{j-1}\left(\frac{t^2}{2} (x_j-\mathcal{E}_{j-1}(x_j)-a_j)^2h\left(tM\right)\right)\nonumber\\
&& \qquad\qquad \qquad(\mbox{Using functional calculus to~} x_j-\mathcal{E}_{j-1}(x_j)-a_j (\leq M) \mbox{~and~} \eqref{MOS4})\nonumber\\
&=&\frac{h(tM)}{2}t^2\mathcal{E}_{j-1}\left(\left(x_j-\mathcal{E}_{j-1}(x_j)-a_j\right)^2\right)\nonumber\\
&=&\frac{h(tM)}{2}t^2 \left(\mathcal{E}_{j-1}((x_j-\mathcal{E}_{j-1}(x_j))^2)+2a_j\mathcal{E}_{j-1}(x_j-\mathcal{E}_{j-1}(x_j))+a_j^2\right)\nonumber\\
&=&\frac{h(tM)}{2}t^2(\mathcal{E}_{j-1}((x_j-\mathcal{E}_{j-1}(x_j))^2)+a_j^2)\nonumber\\
&\leq&\frac{h(tM)}{2}t^2(\sigma_j^2+b_jx_{j-1}+a_j^2).\qquad\qquad\qquad\qquad \qquad (\mbox{by hypothesis (i) })\nonumber\\
&\leq&\frac{h(tM)}{2}t^2(\sigma_j^2+b_jM_{j-1}+a_j^2)\qquad\qquad\qquad\qquad \qquad (\mbox{by } x_{j-1}\leq M_{j-1})
\end{eqnarray}
\noindent \textbf{Step (IV). To establish a recurrence relation.}\\
We have
\begin{eqnarray*}
\tau\left(e^{tx_j}\right)
&\leq&\tau\left(e^{\frac{t\mathcal{E}_{j-1}(x_j)+ta_j}{2}}
e^{-ta_j+\sum_{k=2}^{\infty}\frac{t^k}{k!}\mathcal{E}_{j-1}\left((x_j-\mathcal{E}_{j-1}(x_j)-a_j)^k\right)}e^{\frac{t\mathcal{E}_{j-1}(x_j)+ta_j}{2}}\right)
\quad (\mbox{by~} \eqref{sal})\nonumber\\
&\leq&\tau\left(e^{\frac{t\mathcal{E}_{j-1}(x_j)+ta_j}{2}}
e^{-ta_j+\frac{h(tM)}{2}t^2(\sigma_j^2+b_jM_{j-1}+a_j^2)}e^{\frac{t\mathcal{E}_{j-1}(x_j)+ta_j}{2}}\right)\nonumber\\
&&\qquad (\mbox{since by the functional calculus and \eqref{n1}) ~} x\leq c\Rightarrow e^x\leq e^c,\,c\in\mathbb{R})\nonumber\\
&=&\tau\left(e^{t\mathcal{E}_{j-1}(x_j)+\frac{h(tM)}{2}t^2(\sigma_j^2+b_jM_{j-1}+a_j^2)}\right)\nonumber\\
&=&\exp\left\{\frac{h(tM)}{2}t^2\left(\sigma_j^2+b_jM_{j-1}\right)+a_j^2\right\}\tau\left(e^{t\mathcal{E}_{j-1}(x_j)}\right)\nonumber\\
&=&\exp\left\{\frac{h(tM)}{2}t^2\left(\sigma_j^2+b_jM_{j-1}\right)+a_j^2\right\}\tau\left(e^{tx_{j-1}+t\mathcal{E}_{j-1}(x_j)-tx_{j-1}}\right)\nonumber\\
&\leq&\exp\left\{\frac{h(tM)}{2}t^2\left(\sigma_j^2+b_jM_{j-1}\right)+a_j^2\right\}\tau\left(e^{\frac{tx_{j-1}}{2}}
e^{t\mathcal{E}_{j-1}(x_j)-tx_{j-1}}e^{\frac{tx_{j-1}}{2}}\right)\nonumber\\
&&\hspace{5cm} (\mbox{by inequality \eqref{T1}) }\nonumber\\
&\leq&\exp\left\{\frac{h(tM)}{2}t^2\left(\sigma_j^2+b_jM_{j-1}\right)+a_j^2\right\}\tau\left(e^{tx_{j-1}}\right)\nonumber\\
&&\quad (\mbox{since the inequality }\mathcal{E}_{j-1}(x_j)\leq x_{j-1} \mbox{~yields that~} e^{t\mathcal{E}_{j-1}(x_j)-tx_{j-1}}\leq 1)\nonumber
\end{eqnarray*}
\noindent \textbf{Step (V). To find an upper bound for ${\rm Prob}(x_n \geq \lambda)$.}\\
Assume that $t<3/M$ has been chosen and $\lambda >0$. The Chebyshev inequality \eqref{GH1} yields that

\begin{eqnarray} \label{arash}
{\rm Prob}(x_n \geq \lambda)&\leq& e^{-t\lambda}\tau\left(e^{tx_n}\right)\nonumber\\
&=&e^{-t \lambda}\tau\left(\mathcal{E}_n\left(e^{tx_n}\right)\right)\nonumber\\
&\leq&e^{-t \lambda}\exp\left\{\frac{h(tM)}{2}t^2\left(\sigma_n^2+b_nM_{n-1}\right)+a_n^2)\right\}\tau\left(e^{tx_{n-1}}\right)\nonumber\\
&\leq& e^{-t\lambda}\exp\left\{\frac{h(tM)}{2}t^2)\sum_{j=1}^n(\sigma_j^2+b_jM_{j-1}+a_j^2)\right\}\tau(e^{tx_0})\qquad({\rm inductively})\nonumber\\
&=&\exp\left\{-t\lambda+\frac{h(tM)}{2}t^2\sum_{j=1}^n(\sigma_j^2+b_jM_{j-1}+a_j^2)\right\}\qquad\qquad (\mbox{by~} x_0=0)\nonumber\\
&\leq&\exp\left\{-t\lambda+\frac{t^2}{2(1-tM/3)}\sum_{j=1}^n(\sigma_j^2+b_jM_{j-1}+a_j^2)\right\}\\
&&\qquad\qquad\quad ({\rm Since~for~} \alpha<3, {\rm we~have~}
h(\alpha)\leq \sum_{k=0}^\infty\left(\frac{\alpha}{3}\right)^k=\frac{1}{1-\frac{\alpha}{3}}~~(*))\nonumber\\
&\leq&\exp\left\{-t\lambda+\frac{t^2}{2(1-tM/3)}\sum_{j=1}^n(\sigma_j^2+Db_j+a_j^2)\right\},
\end{eqnarray}
where $D:=\max_{1\leq j\leq n-1}M_j$. 
Now set $t=\frac{\lambda}{\sum_{j=1}^n(\sigma_j^2+Db_j+a_j^2)+(M\lambda/3)}\in (0,3/M)$ to get
$${\rm Prob}\left( x_n \geq \lambda\right)\leq \exp\left\{\frac{-\lambda^2}{2\left(\sum_{j=1}^n(\sigma_j^2+Db_j+a_j^2)+(M\lambda/3)\right)}\right\}\,.$$
Therefore symmetry and the last inequality imply that
$${\rm Prob}\left(| x_n| \geq \lambda\right)\leq 2\exp\left\{\frac{-\lambda^2}{2\left(\sum_{j=1}^n(\sigma_j^2+Db_j+a_j^2)+(M\lambda/3)\right)}\right\}\,.$$

\noindent \textbf{Step (VI). To prove the theorem in the general case}\\
We assume now that $x=(x_j)_{0\leq j\leq n}$ is an arbitrary supermartingale. Since $\mathcal{E}_{j-1}(x_0)=\mathcal{E}_{j-1}(\mathcal{E}_0(x_0))=\mathcal{E}_0(x_0)=x_0$, we infer that $(x_j-x_0)_{0\leq j\leq n}$ is a supermartingale, whose first term is $0$. So we conclude \eqref{MOS2}.
\end{proof}

If we take martingales and put $b_j=0$ in Theorem \ref{main3}, then we get the following Azuma inequality for martingales.
\medskip
\begin{theorem}\label{main2}
Suppose that $x=(x_j)_{0\leq j\leq n}$ is a self-adjoint martingale with respect to a filtration $(\mathfrak{M}_j, \mathcal{E}_j)_{0\leq j\leq n}$ and $dx_j=x_j-x_{j-1}$ is its associated martingale difference such that for some positive constants $a_j, \sigma_j$ and $M$ satisfies
\begin{itemize}
\item [(i)] $\mathcal{E}_{j-1}((dx_j)^2)\leq\sigma_j^2$,
\item [(ii)] $dx_j\leq a_j+M$
\end{itemize}
for all $1\leq j\leq n$. Then
\begin{eqnarray}\label{MOS2}
{\rm Prob}\left(\left|\sum_{j=1}^ndx_j\right|\geq \lambda\right)\leq 2\exp\left\{\frac{-\lambda^2}{2\left(\sum_{j=1}^n(\sigma_j^2+a_j^2\right)+M\lambda/3)}\right\}
\end{eqnarray}
for all $\lambda > 0$.
\end{theorem}


The next corollary reads as follows.

\begin{corollary}\label{EC1}
Suppose that $x=(x_j)_{0\leq j\leq n}$ is a self-adjoint martingale with respect to a filtration $(\mathfrak{M}_j, \mathcal{E}_j)_{0\leq j\leq n}$ such that for some constants $\sigma_j$ and $M$ satisfies
\begin{itemize}
\item [(i)] $\mathcal{E}_{j-1}((dx_j)^2)\leq \sigma_j^2$,
\item [(ii)] $dx_j:=x_j-x_{j-1}\leq M$
\end{itemize}
for $1\leq j\leq n$. Then
\begin{eqnarray}
\tau\left(e^{\lambda (x_n-x_0)}\right)\leq \exp\left\{\frac{\lambda^2 K^2}{2(1-\lambda M/3)}\right\}
\end{eqnarray}
for all $\lambda<\frac{3}{M}$, where $K^2=\sum_{j=1}^n \sigma_j^2$.
\end{corollary}
\begin{proof}
We have
\begin{eqnarray}\label{n2}
\mathcal{E}_{j-1}\left(e^{t(x_j-x_{j-1})}\right)&=&1+\mathcal{E}_{j-1}\left(\sum_{k=2}^{\infty}\frac{t^k}{k!}(x_j-x_{j-1})^k\right)\qquad (\mbox{by~} \mathcal{E}_{j-1}(x_j)=x_{j-1})\nonumber\\
&\leq&\exp\left\{\frac{t^2}{2}h(tM)\sigma_j^2\right\}.\,\,(\mbox{by~}\eqref{n1} \mbox{~for~} a_j=b_j=0, 1\leq j\leq n)
\end{eqnarray}
We deduce form Golden--Thompson inequality that
\begin{eqnarray*}
\tau\left(e^{\lambda (x_n-x_0)}\right)&=&\tau\left(e^{\lambda\sum_{j=1}^n(x_j-x_{j-1})}\right)\\
&\leq&\tau\left(\mathcal{E}_{n-1}\left(e^{\lambda\sum_{j=1}^{n-1}(x_j-x_{j-1})}e^{\lambda(x_n-x_{n-1})}\right)\right)\\
&=&\tau\left(e^{\lambda\sum_{j=1}^{n-1}(x_j-x_{j-1})}\mathcal{E}_{n-1}\left(e^{\lambda(x_n-x_{n-1})}\right)\right)\\
&\leq&\exp\left\{\frac{\lambda^2}{2}h(\lambda M)\sigma_n^2\right\}\tau\left(e^{\lambda\sum_{j=1}^{n-1}(x_j-x_{j-1})}\right)\,\, (\mbox{by inequality~} \eqref{n2})\\
&\leq& \exp\left\{\frac{\lambda^2}{2}h(\lambda M)\sum_{j=1}^n\sigma_j^2\right\}\qquad\qquad\qquad\qquad\qquad (\mbox{inductively})\\
&\leq&\exp\left\{\frac{\lambda^2K^2}{2(1-\lambda M/3)}\right\} \qquad\qquad\qquad\qquad\qquad\qquad (\mbox{by~} (*))
\end{eqnarray*}
for all $\lambda<\frac{3}{M}$.
\end{proof}
In the next result we use a strategy of \cite[Corollary 0.3]{JZ1} to get an estimation of $\left\|\sum_{j-1}^ndx_j\right\|_p$.

\begin{corollary}\label{z123}
Suppose that $x=(x_j)_{0\leq j\leq n}$ is a self-adjoint martingale with respect to a filtration $(\mathfrak{M}_j, \mathcal{E}_j)_{0\leq j\leq n}$ and $dx_j=x_j-x_{j-1}$ is its associated martingale difference such that for some positive constants $\sigma_j$ and $M$ satisfies
\begin{itemize}
\item [(i)] $\mathcal{E}_{j-1}((dx_j)^2)\leq\sigma_j^2$,
\item [(ii)] $dx_j\leq M$
\end{itemize}
for all $1\leq j\leq n$. Then
\begin{eqnarray*}
{\rm Prob}\left(\left|\sum_{j=1}^ndx_j\right|\geq t\right)\leq 2\exp\left\{\frac{-3t^2}{6\sum_{j=1}^n \sigma_j^2+2tM}\right\},
\end{eqnarray*}
and \begin{eqnarray*}
\left\|\sum_{j-1}^ndx_j\right\|_p\leq \sqrt{3p}\left(\sum_{j-1}^n\left\|\mathcal{E}_{j-1}((dx_j)^2)\right\|\right)^{\frac{1}{2}}+
\sqrt{8}p\max_{1\leq j\leq n}\|dx_j\|.
\end{eqnarray*}
for $2\leq p<\infty$.
\end{corollary}
\begin{proof}
The first inequality is an immediate consequence of Theorem \ref{main2}. Now we prove the inequality involving the Schatten norm.\\
Let $K^2=\sum_{j=1}^n\|\mathcal{E}_{j-1}((dx_j)^2)\|$. Note that $\max_{1\leq j\leq n}\|dx_j\|\leq M$. It follows form (\ref{I}) that
\begin{eqnarray*}
\left\|\sum_{j=1}^ndx_j\right\|_p^p&\leq&2p\int_0^{\infty}t^{p-1}\exp\left(\frac{-3t^2}{6K^2+2tM}\right)dt\\
&=&2p\left(\int_0^{\frac{3K^2}{2M}}t^{p-1}\exp\left(\frac{-3t^2}{6K^2+2tM}\right)dt+\int_{\frac{3K^2}{2M}}^{\infty}t^{p-1}
\exp\left(\frac{-3t^2}{6K^2+2tM}\right)dt\right)\\
&\leq&2p\left(\int_0^{\frac{3K^2}{2M}}t^{p-1}\exp\left(\frac{-t^2}{3K^2}\right)dt+\int_{\frac{3K^2}{2M}}^{\infty}t^{p-1}
\exp\left(\frac{-t}{2M}\right)dt\right).
\end{eqnarray*}
By the change of variable $t^2=3K^2r$ and employing $\Gamma(\alpha):=\int_0^\infty e^{-r}r^{\alpha-1}\leq \alpha^{\alpha-1}\,\,(\alpha\geq1)$ we get \begin{eqnarray*}
\int_0^{\frac{3K^2}{2M}}t^{p-1}\exp\left(\frac{-t^2}{3K^2}\right)dt=\frac{1}{2}(3K^2)^{\frac{p}{2}}\int_0^{\frac{3K^2}{4M^2}}e^{-r}r^{\frac{p}{2}-1}dr
\leq \frac{1}{2}3^{\frac{p}{2}}K^p\Gamma\left(\frac{p}{2}\right)
\leq\frac{1}{2}3^{\frac{p}{2}}K^p\left(\frac{p}{2}\right)^{\frac{p}{2}-1}
\end{eqnarray*}
The change of variable $t=2Mr$ yields that
\begin{eqnarray*}
\int_{\frac{3K^2}{2M}}^{\infty}t^{p-1}\exp\left(\frac{-t}{2M}\right)dt=2^pM^p\int_{\frac{3K^2}{4M^2}}^{\infty}r^{p-1}e^{-r}dr
\leq 2^pM^p\Gamma(p)\leq2^pM^pp^{p-1}.
\end{eqnarray*}
Thus we obtain
\begin{eqnarray*}
\left\|\sum_{j=1}^ndx_j\right\|_p^p\leq 2p\left(\frac{1}{2}3^{\frac{p}{2}}K^p\left(\frac{p}{2}\right)^{\frac{p}{2}-1}+2^pM^pp^{p-1}\right).
\end{eqnarray*}
It follows from Minkowski inequality that
\begin{eqnarray*}
\left\|\sum_{j=1}^ndx_j\right\|_p\leq 2^{-1/2+1/p}K\sqrt{3p}+2^{1+1/p}Mp\leq \sqrt{3p}K+2^{3/2}pM\,.
\end{eqnarray*}

\end{proof}

As a consequence we get a noncommutative Bernstein inequality; see \cite[Corollary 2.2.]{GHMO} and \cite[Corllary 0.2]{JZ1}.

\begin{theorem} (Noncommutative Bernstein inequality)\label{Ber}
Let $\mathfrak{N}\subseteq\mathfrak{A}_j(\subseteq\mathfrak{M})$ be order independent over $\mathfrak{N}$. Let $x_j\in\mathfrak{A}_j$ be self-adjoint such that
\begin{itemize}
\item [(i)] $\mathcal{E}_{\mathfrak{N}}(x_j)=0$,
\item [(ii)] $\mathcal{E}_{\mathfrak{N}}(x_j^2)\leq b_j^2$,
\item [(iii)] $\|x_j\|\leq M$,
\end{itemize}
for some $M>0$ and all $1\leq j\leq n$. Then for each $\lambda\geq0$,
\begin{eqnarray*}
{\rm Prob}\left(\sum_{j=1}^nx_j\geq \lambda\right)\leq \exp\left(-\frac{\lambda^2}{2b^2+(2/3)\lambda M}\right),
\end{eqnarray*}
where $b^2=\sum_{j=1}^nb_j^2$.
\end{theorem}
\begin{proof}
Put $S_0:=0$ and $S_j:=\sum_{k=1}^jx_k\,\,(1\leq j\leq n)$. As one can see from the proof of the noncommutative Hoefdding inequality \ref{H} that $(S_j)_{0\leq j\leq n}$ is a martingale. Since $\mathfrak{N}\subseteq\mathfrak{A}_j\subseteq\mathfrak{M}$ is order independent over $\mathfrak{N}$ we have
$\mathcal{E}_{j-1}((dS_j)^2)=\mathcal{E}_{\mathfrak{N}}(x_j^2)\leq b_j^2$. In addition, $dS_j\leq \|dS_j\|\leq M$.
Now the required inequality is deduced from Corollary \ref{z123} with the $S_n$ instead of $x_n$.
\end{proof}

\begin{corollary}
Suppose that $x=(x_j)_{0\leq j\leq n}$ is a self-adjoint martingale with respect to a filtration $(\mathfrak{M}_j, \mathcal{E}_j)_{0\leq j\leq n}$ such that for some constants $ \sigma_j$ and $M_j$ satisfies
\begin{itemize}
\item [(i)] $\mathcal{E}_{j-1}((dx_j)^2)\leq\sigma_j^2$,
\item [(ii)] $dx_j\leq M_j$
\end{itemize}
for $1\leq j\leq n$. Then
\begin{eqnarray}\label{MOS2}
{\rm Prob}\left(\left|\sum_{j=1}^ndx_j\right|\geq \lambda\right)\leq 2\exp\left\{\frac{-\lambda^2}{2\sum_{j=1}^n\sigma_j^2+\sum_{M_j>M}(M_j-M)^2+
M\lambda/3}\right\}
\end{eqnarray}
for any $M$.
\end{corollary}
\begin{proof}
It follows from Theorem \ref{main2} by choosing
\begin{eqnarray*}
a_j=\begin{cases}
\begin{array}{ll}
0&{\rm if}~~ M_j\leq M \\
M_j-M & {\rm if}~~ M_j\geq M.
\end{array}
\end{cases}
\end{eqnarray*}
\end{proof}


\end{document}